\documentclass[11pt]{amsart}

\usepackage{amsmath, amsfonts,amscd,flafter,epsf,amssymb,epsfig,subfigure,bm,url, mathrsfs, amsthm, comment}
\usepackage{xcolor}
\newtheorem{theorem}{Theorem}[section]

\newtheorem{lemma}[theorem]{Lemma}
\newtheorem{corollary}[theorem]{Corollary}

\newtheorem*{theorem*}{Theorem}{\bf}{\it}
\newtheorem*{proposition*}{Proposition}{\bf}{\it}
\newtheorem*{observation*}{Observation}{\bf}{\it}
\newtheorem*{lemma*}{Lemma}{\bf}{\it}

\theoremstyle{definition}
\newtheorem{definition}[theorem]{Definition}

\theoremstyle{remark}
\newtheorem{remark}[theorem]{Remark}

\newcommand{\p}{\mathcal P}

\usepackage{pgfplots,tikz}
\usepackage{xkeyval,tkz-base,tikz-3dplot}
\usepackage{tqft}
\usetikzlibrary{arrows, calc, shapes, automata, backgrounds, petri, positioning,fadings,decorations.pathreplacing, spy, intersections}
\usetikzlibrary{external}
\usetikzlibrary{fillbetween}
\usetikzlibrary{decorations.markings}
\usetikzlibrary{lindenmayersystems}
\tikzset{
	external/system call={
	xelatex \tikzexternalcheckshellescape
	-halt-on-error -interaction=batchmode --shell-escape --enable-write18
	-jobname "\image" "\texsource"}
}
\usepgfplotslibrary{external} 
\def\nextAngle{0}
\tikzset{
	next angle/.style={
		in=#1+180,
		out=\nextAngle,
		prefix after command= {\pgfextra{\def\nextAngle{#1}}}
	},
	start angle/.style={
		out=#1,
		nangle=#1,
	},
	nangle/.code={
		\def\nextAngle{#1}
	}
}

\newcommand{\abs}[1]{\left| #1 \right|}

\newcommand{\set}[1]{\left\{ #1 \right\}}
\newcommand{\bracket}[1]{\left( #1 \right)}

\newcommand{\calx}{\mathcal{X}}
\newcommand{\caly}{\mathcal{Y}}

\newcommand{\bbh}{\mathbb{H}}

\renewcommand{\tilde}{\widetilde}

\def\XXint#1#2#3{{\setbox0=\hbox{$#1{#2#3}{\int}$ }
\vcenter{\hbox{$#2#3$ }}\kern-.6\wd0}}

\begin{document}
\title[]{Closed geodesics on hyperbolic surfaces with few intersections}

\author{Wujie Shen}
\address{Department of Mathematics, Tsinghua University}
\email{shenwj22@mails.tsinghua.edu.cn}

\begin{abstract}
 We prove that, if a closed geodesic $\Gamma$ on a complete finite-type hyperbolic surface has at least 2 self-intersections, then the length of $\Gamma$ is bounded below by $2\log(5+2\sqrt6)$, This lower bound is sharp and is achieved by a corkscrew geodesic on a thrice-punctured sphere.
\end{abstract}

\maketitle

\section{Introduction} 
The study of closed geodesics on hyperbolic surfaces constitutes a fundamental topic in hyperbolic geometry, offering profound insights into geometric structures and surface dynamics. Within this field, nonsimple closed geodesics are of particular interest, as their investigation is deeply intertwined with diverse areas of mathematics. These include Teichmüller theory, spectral theory, the geometry of moduli spaces, and group actions on hyperbolic spaces (e.g., \cite{PP1,YHW1,YHW2}). Consequently, understanding the distribution and lengths of nonsimple closed geodesics is essential for unraveling the global structure and dynamical properties of hyperbolic surfaces.

Let $M_k$ represent the infimal length of a nonsimple closed geodesic with self-intersection number at least $k$ among all finite-type hyperbolic surfaces. By studying $M_k$ for different values of $k$, we can gain insight into the prevalence and behavior of nonsimple closed geodesics on hyperbolic surfaces. A natural goal is to find the growth rate or the exact value of $M_k$.

There has been extensive research on the problem so far. When $k = 1$, Hempel demonstrated in \cite{H2797} that the length of a nonsimple closed geodesic has a universal lower bound of $2\log(1+\sqrt{2})$, while Yamada showed in \cite{Y2799} that $2 \cosh^{-1}(3)=4\log(1+\sqrt{2})$ is the best possible lower bound and is achieved on a pair of pants with 3 ideal punctures. Basmajian also demonstrated in \cite{B2794} that a nonsimple closed geodesic has a \emph{stable neighborhood}, and the length of a closed geodesic increases indefinitely as its self-intersection number grows (\cite[Corollary 1.2]{B2794}). In \cite{B2795}, Baribaud calculated the minimal length of geodesics with given self-intersection number or given homotopy types on pairs of pants.

Let $\omega$ represent a closed geodesic or a geodesic segment on a hyperbolic surface $\Sigma$, which can be described as a local isometry $\phi$ from $S^1$ or segment I to $\Sigma$. The length of $\omega$ is denoted as $\ell({\omega})$.

\begin{definition}
Its self-intersection number is represented as $\abs{\omega\cap\omega}$. It counts the intersection points of $\omega$ with multiplicity that an intersection point with $n$ preimages of $\phi$ contribute $\binom{n}{2}$ to $\abs{\omega\cap\omega}$. 
\end{definition}

 First, when $k\rightarrow\infty$, Basmajian showed (\cite[Corollary 1.4]{B2803}) that
\begin{equation}\label{eqn:growth_bas}
\tfrac12\log \frac{k}2\leqslant M_k \leqslant 2\cosh^{-1}(2k+1)\asymp 2\log k
\end{equation}
The notation $f(k)\asymp g(k)$ means that ${f(k)}/{g(k)}$ is bounded from above and below by positive constants. We can conjecture that, when $k\geqslant1$,
\begin{equation}\label{eqn:length_asymptotic}
M_k=2\cosh^{-1}(1+2k)=2\log(1+2k+2\sqrt{k^2+k})
\end{equation}
and the equality holds when $\Gamma$ is a corkscrew geodesic(See definition below) on a thrice-punctured sphere. In other words, on any finite-type hyperbolic surface, a nonsimple closed geodesic with a self-intersection number of at least $k$ must have a minimum length of $2\cosh^{-1}(1+2k)$ and the bound is sharp.

In \cite[Theorem 1.1]{Y2800} Shen-Wang improved the lower bound of $M_k$, that $M_k$ has explicit growth rate $2\log{k}$, and for a closed geodesic of length $L$, the self intersection number is no more than $9L^2e^{\frac{L}{2}}$. The exact value for $M_k$ for sufficiently large $k$ is computed in \cite[Theorem 1.1]{Y2801}, proved the conjecture holds when $k>10^{13350}$. In \cite{Y2806} the lower bound $k>10^{13350}$ is refined to $k>1750$ using another method.

However, to the best of the author's knowledge, for small $k$, even $k=2$ we cannot compute the exact value of $M_k$. In the present paper we give an answer:

\begin{theorem}\label{main}
    When $k=2$, $M_2=2\log(5+2\sqrt6)$, and the lower bound is sharp and attained on a corkscrew geodesic on a thrice punctured sphere.
\end{theorem}

Note that Theorem \ref{main} can be generalized to general orientable finite-type hyperbolic surfaces, possibly
with holes and geodesic boundaries, since they can be doubled to get a surface as in Theorem \ref{main}.

\subsection*{Acknowledgement} 
We would be thankful to many people for giving some discussions in this paper, especially Yi Huang and Zhongzi Wang for reading the first draft of the paper and giving some helpful comments.

\section{A generalization of the collar lemma}\label{collar_lemma}

In this section, we present a generalization of the collar lemma. In contrast to the standard lemma, we show that a collar remains embedded even when the widths of the two half-collars adjacent to a short closed geodesic are not equal. The generalized collar lemma will be used in the proof of Theorem \ref{LL106}. By increasing the width of the collar on one side, a greater portion of the geodesic $\Gamma$ is forced into a region whose geometry is well-understood. The collar lemma (\cite[Lemma 13.6]{FM2012}) states:
\begin{lemma}\label{collar}
for any simple closed geodesic $c$ in $\Sigma$, $$N(c)=\set{p\in\Sigma\,:\,d(p,c)< w(\ell(c))}$$
is an embedded annulus, where $w(x)$ is defined by
$$w(x):=\sinh^{-1}\bracket{\frac{1}{\sinh({x}/{2})}}.$$
Moreover if $c_1,c_2$ are two disjoint simple closed geodesics in $\Sigma$, $N(c_1)\cap N(c_2)=\emptyset$.
\end{lemma}

Let $\Sigma'$ be the surface constructed by cutting along curve $c$ on surface $\Sigma$, and denote the boundary components obtained from cutting $c$ by $c'$ and $c''$. Let $\phi':\Sigma'\rightarrow\Sigma$ denote the gluing map by reattaching the boundary components $c'$ and $c''$ back to the curve on $\Sigma$. In $\Sigma'$ there exists a \emph{maximal collar} $N'(c')$  that $\delta(c')>0$ is the maximum real number for which 
$$\tilde{N}'(c')=\set{p\in\Sigma'\,:\,d(p,c')<\delta(c')}$$
 is an embedded annulus in $\Sigma'$. $N'(c)=\phi'(N'(c'))$ is an embedded annulus(a half collar) with $c$ as one of its boundary components.
We define 
$$w_1(x)=\log\bracket{\frac{e^{x/4}+1}{e^{x/4}-1}}=\sinh^{-1}\bracket{\frac{1}{\sinh(x/4)}}$$
for $x>0$. Note that for $x\leqslant\log(5+2\sqrt6)<2.3$ we have $w_1(x)<2w(x)$, and for all $x>0$ we have $w_1(x)>w(x)$.

Since $w_1(x)>w(x)$ for $x>0$, we have
$$N''(c'')=\set{p\in\Sigma'\,:\,d(p,c'')<2w(\ell(c))-w_1(\ell(c))}\subseteq\Sigma'$$
is either empty or an embedded annulus, and define $N''(c)=\phi'(N''(c''))\subseteq\Sigma$, either empty or an embedded annulus also. Define 
$$N_1(c)=N'(c)\cup N''(c)$$

In this section, we will prove the generalized collar lemma for closed geodesic $c\subseteq\Sigma$:

\begin{lemma}[Generalized collar lemma]
 Notations as above. $N_1(c)$ is an embedded annulus in $\Sigma$.
\end{lemma}

\begin{figure}[htbp]
\centering
\tikzexternaldisable
\begin{tikzpicture}[declare function={
	R = 9;
	f(\x) = R+.5-sqrt(R*R-\x*\x);
	ratio = 3;
}]
\def\x{1}
\draw [thick] (-3.2,{f(3.2)}) arc ({270-asin(3.2/R)}:{270+asin(4/R)}:{R} and {R});
\draw [thick] (-3.2,{-f(3.2)}) arc ({90+asin(3.2/R)}:{90-asin(4/R)}:{R} and {R});


\foreach \a/\b in {0/blue, 4/red} {
	\draw [color=\b, thick] (\a,{-f(\a)}) arc (-90:90:{f(\a)/ratio} and {f(\a)});
	\draw [dashed, color=\b, thick] (\a,{f(\a)}) arc (90:270:{f(\a)/ratio} and {f(\a)});
}

\foreach \a/\b in {-3.2/red} {
	\draw [color=\b, thick] (\a,{-f(\a)}) arc (-90:90:{f(\a)/ratio} and {f(\a)});
	\draw [color=\b, thick] (\a,{f(\a)}) arc (90:270:{f(\a)/ratio} and {f(\a)});
}


\path (-2,0) node[circle, inner sep=1pt, label={[label distance=.4em, anchor=center]0:$N''(c)$}]{};

\path (2,0) node[circle, inner sep=1pt, label={[label distance=.4em, anchor=center]0:$N'(c)$}]{};

\path ({f(0)/ratio},0) node[circle, inner sep=1pt, label={[label distance=.4em, anchor=center]0:$c$}]{};
  			
\end{tikzpicture}

\caption{\label{fig:special_case}
The generalized collar $N_1(c)$}
\end{figure}

First is the half-collar is an embedded annulus:

\begin{lemma}
 If $\ell(c')<2.3$, then we have
 $\delta(c')\geqslant w_1(\ell(c'))$.
 In other words, the collar $$N'(c')=\set{x\in\Sigma'\,:\,d(x,c)< w_1(\ell(c))}$$ is an embedded annulus.
\end{lemma}

\begin{proof}

We know $\tilde{N}'(c')$ has two boundary components. Define  $\overline{\tilde{N}'(c')}$ to be the closure of $\tilde{N}'(c')$. If $\overline{\tilde{N}'(c')}\cap c''\neq\emptyset$, then the result holds, as $w_1(\ell(c'))<2w(\ell(c'))$, which implies $N'(c')\cap c''=\emptyset$.

If $\overline{\tilde{N}'(c')}$ does not intersect with $c''$, then one of the components of $\partial N'(c)$ is $c$ and the other is a curve tangent to itself due to the maximality of $\delta(c')$. Assuming $q\in\partial N'(c)$ is one of the tangent points. Therefore there exists a shortest simple geodesic $\gamma$ contained in $\overline{N'(c)}\subseteq\Sigma$ which is joined by two different shortest geodesic segments of length $\delta(c')$ connecting $q$ and $c$ in $N'(c)$. $\gamma$ has endpoints $r_1,r_2\in c$, and is perpendicular to $c$ at $r_1,r_2$.

Suppose $c\setminus\{r_1,r_2\}$ contains two arcs $c_1'$ and $c_2'$, and $c_1'\cup\gamma$, $c_2'\cup\gamma$ are simple closed curves, hence freely homotopic to simple closed geodesics $c_1''$ and $c_2''$, then $c,c_1'',c_2''$ are boundary components of a pair of pants $\Sigma_0\subseteq\Sigma$, since $c_1''\cup\gamma$ and $c_2''\cup\gamma$ cannot form bigons, $c_1''\cap\gamma=c_2''\cap\gamma=\emptyset$, we have $\gamma\subseteq\Sigma_0$.

\begin{figure}[htbp]
\centering
\tikzexternaldisable
\begin{tikzpicture}
\def\x{4}
\coordinate (O) at (0,0);
\coordinate (P) at (-2.2*\x,0);
\coordinate (Q) at (2.2*\x,0);
\coordinate (K) at (0, 1.3*\x);

\draw[color=black]
  (4,0) arc (0:180:4);
\draw[color=cyan, thin] (180:\x) to (0:\x);

\pgfmathsetmacro{\anglePG}{\pgfmathresult}
\draw[name path=arc_GSS, color=cyan, thin]
  (-1,0) arc (0:28.07:7.5);
\draw[name path=arc_GS, color=cyan, thin]
  (1,0) arc (180:151.93:7.5);
 
\pgfmathsetmacro{\angleKE}{\pgfmathresult}
\draw[name path=arc_EF, color=cyan, thin]
  (0,3) arc (270:343.74:1.1667);
\draw[name path=arc_EF, color=cyan, thin]
  (0,3) arc (270:196.26:1.1667);

\draw[name path=arc_EF, color=cyan, thin]
  (0,2.4) arc (270:331.93:2.1333);
\draw[name path=arc_EF, color=cyan, thin]
  (0,2.4) arc (270:208.07:2.1333);

\path (0,2.4) node[circle, inner sep=1pt, label={[shift={(270:0.3)}, anchor=center]{$\omega_3'$}}]{};
\path (0,3) node[circle, inner sep=1pt, label={[shift={(90:0.3)}, anchor=center]{$\omega_3$}}]{};

\draw[color=red, thick] (-1,0) to (1,0);
\draw[name path=arc_GS, color=red, thick]
  (-1,0) arc (0:22:7.5);
\draw[name path=arc_GS, color=red, thick]
  (1,0) arc (180:158:7.5);
\draw[name path=arc_EF, color=red, thick]
  (0,3) arc (270:322:1.1667);
\draw[name path=arc_EF, color=red, thick]
  (0,3) arc (270:218:1.1667);
\draw[color=red, thick] (-0.9,3.4) to (-1.56,2.77);
\draw[color=red, thick] (0.9,3.4) to (1.56,2.77);

\path (-1,3.7) node[circle, inner sep=1pt, label={[shift={(135:0.45)}, anchor=center]{$Q_1$}}]{};
\path (1,3.7) node[circle, inner sep=1pt, label={[shift={(45:0.45)}, anchor=center]{$Q_2$}}]{};
\path (-1,0) node[circle, fill, inner sep=1pt, label={[shift={(235:0.45)}, anchor=center]{$A_1$}}]{};
\path (1,0) node[circle, fill, inner sep=1pt, label={[shift={(305:0.45)}, anchor=center]{$A_2$}}]{};
\path (-1.8,3.3) node[circle, inner sep=1pt, label={[shift={(235:0.45)}, anchor=center]{$P_1$}}]{};
\path (1.8,3.3) node[circle, inner sep=1pt, label={[shift={(305:0.45)}, anchor=center]{$P_2$}}]{};

\path (O) node[circle, fill, inner sep=1pt, label={[shift={(270:0.3)}, anchor=center]{$O$}}]{};
\path (-1.3,2) node[circle, inner sep=1pt, label={[shift={(245:0.3)}, anchor=center]{$\omega_1$}}]{};
\path (1.3,2) node[circle, inner sep=1pt, label={[shift={(305:0.3)}, anchor=center]{$\omega_2$}}]{};

\end{tikzpicture}
\caption{\label{fig:boundary_length_annulus000000}
A hexagon of $\Sigma_0$}
\end{figure}
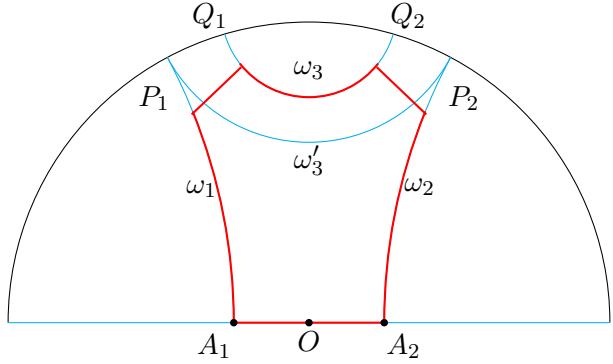

Next we prove $\ell(\gamma)\geqslant w_1(x)$. $\Sigma_0$ can be constructed by gluing two hexagons along three pairwise-nonadjacent boundary segments. In the Poincar\'{e} disk model of $\bbh^2$, let $\bbh^1$ be the horizontal line. Let $A_1,A_2\subseteq\bbh^1$, $A_1(-r,0),A_2(r,0)$ satisfying $\log\frac{1+r}{1-r}=\frac{\ell(c)}4$. $\omega_1,\omega_2$ are geodesics in $\bbh^2$ passes $A_1,A_2$ and perpendicular to $\bbh^1$. $P_1(-\frac{2r}{1+r^2},\frac{1-r^2}{1+r^2})$, $P_2(\frac{2r}{1+r^2},\frac{1-r^2}{1+r^2})$ are one of endpoints of $\omega_1,\omega_2$. $\omega_3'$ is the geodesic connecting $P_1,P_2$. Let $\omega_3$ be the geodesic $Q_1Q_2$ where $Q_1,Q_2$ are both on the infinity boundary, $\omega_3\cap\omega_1=\emptyset$ and $\omega_3\cap\omega_2=\emptyset$. For $j=1,2$ assume $\omega_j'$ be the unique geodesic segment perpendicular to $\omega_j$ and $\omega_3$. Then $A_1A_2, \omega_1, \omega_1', \omega_3, \omega_2', \omega_2$ constitute the boundary of a hexagon, and $\Sigma_0$ is constructed by gluing 2 copies of the hexagon along $\omega_1, \omega_2, \omega_3$, as in Figure \ref{fig:boundary_length_annulus000000}.

Since $\gamma\subseteq\Sigma_0$, then 
\begin{align*}
 \ell(\gamma)&\geqslant 2d(\bbh^1,\omega_3)\geqslant 2d(\bbh^1,\omega_3')=2\log\bracket{1+\frac{1-r}{1+r}}-2\log\bracket{1-\frac{1-r}{1+r}}\\
&=2\log\frac{e^{{\ell(c)}/4}+1}{e^{{\ell(c)}/4}-1}=2w_1(\ell(c))
\end{align*}

\end{proof}

\begin{proof}
 If not, then there exists $y\in N'(c)\cap N''(c)$, connecting $c$ and $y$ by shortest geodesic $\zeta$ in $N'(c)$ and $\zeta'$ in $N''(c)$, and $\ell(\zeta)+\ell(\zeta')<2w(\ell(c))$, hence for all $z\in\zeta\cup\zeta'$, $d(z,c)\leqslant\frac12(\ell(\zeta)+\ell(\zeta'))<w(\ell(c))$, hence $\zeta\cup\zeta'\subseteq N(c)$. $\zeta\cup\zeta'$ is a curve that is piecewisely geodesic, and only intersects $c$ at endpoints so it lies on one side of $N(c)$, which is defined in Lemma \ref{collar}, but this is impossible since $N'(c)$ and $N''(c)$ are on the different side of $c$.
\end{proof}

Similarly, for the ideal punctures of $\Sigma$, we obtain a similar result. In the case where $\Sigma$ has punctures, we consider the universal covering $p:\mathbb{H}^2\rightarrow\Sigma$, where $\mathbb{H}^2$ represents the hyperbolic plane. Each puncture has a neighborhood whose boundary lifts to a union of horocycles that can intersect at most at points of tangency. This type of neighborhood is known as a \emph{horocycle neighborhood} of the surface.

In the upper half-plane model for $\mathbb{H}^2$, let $\Gamma$ be a cyclic group generated by a parabolic isometry of $\mathbb{H}^2$ that fixes the point $\infty$. Let $H_t = \{(x, y)\in \mathbb{H}^2 | y \geqslant t\}$ be a horoball. Each cusp can be modeled as $H_t/\Gamma$ for some $c$ up to isometry and is diffeomorphic to $S^1 × [c, \infty)$, so that each circle $S^1 × {t}$ with $t \geqslant c$ is the image of a horocycle under $p$. Each circle is also referred to as a horocycle. The circle $S^1 × {s}$ with $s \geqslant t$ is called an Euclidean circle. A cusp is considered maximal if it lifts to a union of horocycles with disjoint interiors, and there exists at least one point of tangency between different horocycles.

\begin{lemma}[Adams, \cite{A2817}]\label{lemma:cusp_maximal_area}
For an orientable, metrically complete hyperbolic surface, the area with a maximal cusp is at least $4$. The lower bound $4$ is realized only in an ideal pair of pants.
\end{lemma}

\begin{remark}
If $c$ is a cusp, we define $N'(c)$ as the union of horocycles of $c$ whose lengths are less than 4. This forms an embedded cylinder.
\end{remark}

\section{geodesics on pair of pants}\label{geodesic_pants}

In this section, we outline the classification of closed geodesics on a pair of pants from \cite{B2795}. In the subsequent proof, we will demonstrate that the shortest closed geodesic with a self-intersection number of 2 must lie in some pair of pants, from which the proof of Theorem \ref{main} will follow directly. The \emph{corescrew geodesic} is a typical class of special geodesics.

\begin{definition}[Corkscrew geodesic, \cite{Y2801}]
A \emph{corkscrew geodesic} refers to a geodesic within the homotopy class outlined in Figure \ref{fig:corkscrew}. This geodesic comprises a curve formed by the e concatenation of a simple arc and another that winds $k$-times around a boundary.
\end{definition}

\begin{figure}[htbp]
\centering
\tikzexternaldisable	
\begin{tikzpicture}[declare function={
	R = 42;
	f(\x) = 42.5-sqrt(R*R-\x*\x);
	ratio = 3;
}]
\def\x{2.5}
\draw [thick] (-4.5,2.8) arc (210:270:{5} and {5});
\draw [thick] (-4.5,-2.8) arc (150:90:{5} and {5});


\draw [thick] (-5,2.5) arc (30:-30:{5} and {5});

\draw[smooth, color=blue, thick] (-1.3,-0.4) 
    to [start angle=180, next angle=150] (-2.3, 0.8) coordinate (x);
\draw[smooth, color=blue, thick, dashed] (-2.3, 0.8) 
    to [start angle=150, next angle=90] (-4.33,0.5) coordinate (x);
\draw[smooth, color=blue, thick] (-4.33, 0.5) 
    to [start angle=90, next angle=-30] (-3,1.2) coordinate (x);
\draw[smooth, color=blue, thick, dashed] (x) 
    to [start angle=-30, next angle=20] (-1.7,-0.5) coordinate (x);
\draw[smooth, color=blue, thick] (x) 
    to [start angle=20, next angle=70] (-1.5,0) 
    to [next angle=0] (-1.3,0.4) coordinate (x);
\draw[smooth, color=blue, thick, dashed] (x) 
    to [start angle=0, next angle=270] (-1.1,0)
    to [next angle=180] (-1.3,-0.4) coordinate (x);

\end{tikzpicture}
\caption{\label{fig:corkscrew}
A corkscrew geodesic of $k=2$}
\end{figure}
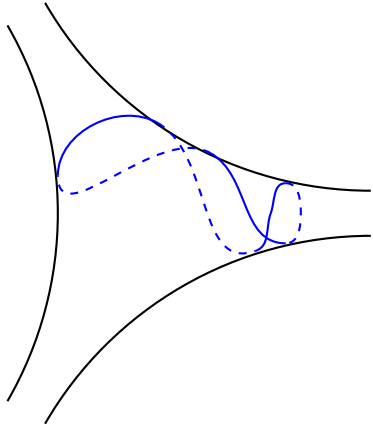

The thrice punctured sphere $\Sigma_{0,3}$ features a distinct hyperbolic structure, characterized by a corkscrew geodesic. This geodesic is formed by the combination of a simple arc and another arc that wraps around a cusp twice, resulting in 2 self-intersections. This can be seen in Figure 2, where the corkscrew geodesic has a length of $2\cosh^{-1}(5)=2\log(5+2\sqrt6)$.

If the geodesic $\Gamma$ is contained in a pair of pants, then we use the following results in \cite{B2795}. Assume a pair of pants $P$ with geodesic boundaries $\gamma_1,\gamma_2,\gamma_3$. For $i=1,2,3$.

 \begin{theorem}
 Let $\Gamma_{m,n}$ be the unique closed geodesic in $P$ that has homotopy type of a curve winding around $\gamma_1$ $m$ times and then winding around $\gamma_2$ $n$ times. The length of $\Gamma_{m,n}$ can be computed as follows:

 \begin{align}\label{length_pants}
  \cosh\frac{\ell(\Gamma_{m,n})}2=\frac{s_{1,m}}{s_1}\frac{s_{2,n}}{s_2}(c_3+c_1c_2)+c_{1,m}c_{2,n}
 \end{align}
 where
 $$ c_i=\cosh\frac{\ell(\gamma_i)}2\ \ \  c_{i,n}=\cosh\frac{n\ell(\gamma_i)}2\qquad s_i=\sinh\frac{\ell(\gamma_i)}2\ \ \  s_{i,n}=\sinh\frac{n\ell(\gamma_i)}2 $$
\end{theorem}
Note that when $P$ is an ideal pair of pants, i.e. some of $\gamma_1,\gamma_2,\gamma_3$ become cusps as their lengths tend to 0, then $c_i,c_{i,n},s_i,s_{i,n}$ can be defined, and if $\gamma_1$ becomes cusp, $\frac{s_{1,m}}{s_1}$ can be altered by $$\lim_{\ell(\gamma_1)\rightarrow0+}\frac{\sinh\frac{m\ell(\gamma_1)}2}{\sinh\frac{\ell(\gamma_1)}2}=m$$
in (\ref{length_pants}). 
And if $\gamma_2$ becomes cusp, similarly $\frac{s_{2,n}}{s_2}$ can be altered by $n$ in (\ref{length_pants}).

\begin{corollary}\label{pants2}
 Suppose $m,n$ are positive integers and $m+n\geqslant3$. $P$ is a pair of pants or an ideal pair of pants. $\Gamma_{m,n}$ is a closed geodesic in $P$ as above. Then
 $$\ell(\Gamma_{m,n})\geqslant2\log(5+2\sqrt6)$$
 Equality holds if and only if $P$ is a thrice punctured sphere and $(m,n)$ is equal to $(1,2)$ or $(2,1)$.
\end{corollary}

\begin{proof}
 Since $c_i=\cosh\frac{\ell(\gamma_i)}2\geqslant1$ and similarly $c_{i,m}\geqslant1$, and when $\ell(\gamma_1)>0$
 \begin{align*}
 \frac{s_{1,m}}{s_1}&=\frac{\sinh\frac{m\ell(\gamma_1)}2}{\sinh\frac{\ell(\gamma_1)}2}=\frac{\alpha^{m}-\alpha^{-m}}{\alpha-\alpha^{-1}}=\sum_{k=0}^{m-1}\alpha^{m-1-2k}\\
 &=\frac12\sum_{k=0}^{m-1}(\alpha^{m-1-2k}+\alpha^{2k+1-m})>m
 \end{align*}
 for $\alpha=e^{\ell(\gamma_1)/2}$. Similarly, when $\ell(\gamma_2)>0$, we have $\frac{s_{2,n}}{s_2}>n$. Hence,we have
 $$\cosh\frac{\ell(\Gamma_{m,n})}2=\frac{s_{1,m}}{s_1}\frac{s_{2,n}}{s_2}(c_3+c_1c_2)+c_{1,m}c_{2,n}\geqslant2mn+1\geqslant5$$
 Hence $\ell(\Gamma_{m,n})\geqslant2\log(5+2\sqrt6)$. The equality holds if and only if $\frac{s_{1,m}}{s_1}=m$, $\frac{s_{2,n}}{s_2}=n$, and $c_3=c_1=c_2=c_{1,m}=c_{2,n}=1$, that is, $\ell(\gamma_1)=\ell(\gamma_2)=\ell(\gamma_3)=0$, and $mn=2$.
\end{proof}

\section{Winding number of geodesic segments in a collar}\label{winding_number}

In this section, we define the \emph{winding number} for geodesics within a collar and a cusp, which establishes a relationship between the number of self-intersections of the geodesic and its length. Let $c\in\Sigma$ be a simple closed geodesic, and $w>0$. This section we always assume the collar
$$N(c)=\set{x\in\Sigma:d(x,c)< w}$$
is an embedded annulus. For cusps, assume $c$ is a puncture and $N(c)$ is defined as the cusp neighborhood with boundary horocycle of length 4. Suppose $\delta$ represents a geodesic segment in $N(c)$ with endpoints $x_1$ and $x_2$ on the same component of $\partial N(c)$.

 We define the \emph{winding number} $W(\delta)$ of the arc $\delta$ as follows. The definitions are similar to \cite{Y2805}.

\begin{enumerate}
\item
When $c$ is a closed geodesic, every point of $\delta$ orthogonally projects to its nearest point on $c$. The winding number of $\delta$ can be calculated as the length of the projection of $\delta$ divided by the length of $c$.

\item
Similarly, when $c$ is a cusp, every point of $\delta$ is orthogonally projects to its nearest point on the length $h$ horocycle, here $h>0$ sufficiently small. The winding number of $\delta$ is the length of the projection of $\delta$ divided by $h$. Note that the definition is independent of $h$.
\end{enumerate}

\begin{lemma}\label{winding001}
 If $c$ is a closed geodesic, then 
$$\ell(\delta)=2\sinh^{-1}\bracket{\sinh\frac{W(\delta)\ell(c)}2\cdot\cosh{w}}$$
\end{lemma}

\begin{proof}
The universal covering $p:\bbh^2\to\Sigma$ from the Poincar\'e disk $\bbh^2$ to $\Sigma$ is locally isometric. Let $\widetilde{\delta}$, $\widetilde{c}$ be a lift of $\delta$ and $c$. The connected component $\widetilde{N(c)}$ of $p^{-1}(N(c))$ containing $\widetilde{\delta}$ is a universal cover of the annulus $N(c)$. 
Let $\widetilde{x}_1$ and $\widetilde{x}_2$ be lifts of $x_1$,$x_2$ in $\widetilde{N(c)}$. Let $\widetilde{\eta}_1$ and $\widetilde{\eta}_2$ be the shortest geodesics from $\widetilde{x}_1$ and $\widetilde{x}_2$ to $\widetilde{c}$ respectively. Then $\eta_1:=p\circ\widetilde{\eta}_1$, $\eta_2=p\circ\widetilde{\eta}_2$ is a geodesic connecting $x$ and $c$ and $\widetilde{\eta}_1$ and $\widetilde{\eta}_2$ are both perpendicular to $c$. Let $\widetilde{y}_1$ and $\widetilde{y}_2$ be the two feet. Without loss of generality, we may assume $\widetilde{c}$ is the horizontal diameter of $\bbh^2$ and the origin $O$ is the middle point of $\widetilde{y}_1$ and $\widetilde{y}_2$, as illustrated in Figure \ref{fig:boundary_length_annulus}.

The geodesic $\widetilde{\delta}$ between $\widetilde{x}_1$ and $\widetilde{x}_2$, $\widetilde{\eta}_1$, $\widetilde{\eta}_2$ and the geodesic $\delta_1$ between $\widetilde{y}_1$ and $\widetilde{y}_2$ form a Saccheri quadrilateral, and half of it is a Lambert quadrilateral. Note that $d(\widetilde{y}_1, \widetilde{y}_2)=W(\delta)\ell(c)$. The property of the Lambert quadrilateral gives
\begin{equation*}\label{eqn:ends_annulus_length_1}
\sinh\bracket{\frac{\ell(\delta)}2}=\sinh\bracket{\frac{W(\delta)\ell(c)}2}\cosh\bracket{\ell\bracket{\widetilde{\eta}_1}}.
\end{equation*}
\end{proof}

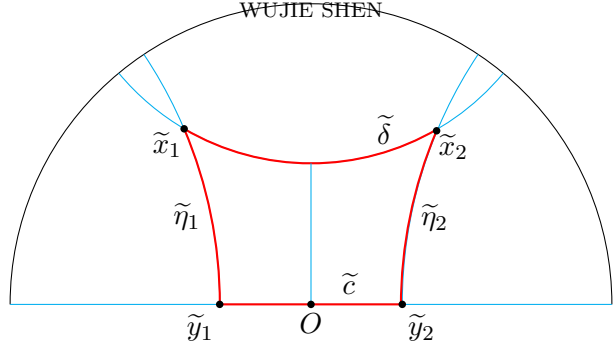
\begin{figure}[htbp]
\centering
\tikzexternaldisable
\begin{tikzpicture}
\def\x{4}
\coordinate (O) at (0,0);
\coordinate (P) at (-1.8*\x,0);
\coordinate (Q) at (1.8*\x,0);
\coordinate (K) at (0, 1.3*\x);

\begin{pgfinterruptboundingbox}
\draw[name path=circle_O, color=black] (4,0) arc (0:180:4);
\path[name path=PQ] (P) to (Q);
\coordinate[name intersections={of=circle_O and PQ, by={N,M}}];

\path[name path=circle_P] let
  \p1 = ($ (M) - (P) $),  \n1 = {veclen(\x1,\y1)},
  \p2 = ($ (N) - (P) $),  \n2 = {veclen(\x2,\y2)} in
  (P) circle ({sqrt(\n1)*sqrt(\n2)});
\coordinate[name intersections={of=circle_P and PQ, by={A}}];
\coordinate[name intersections={of=circle_P and circle_O, by={G,S}}];

\path[name path=circle_Q] let
  \p1 = ($ (M) - (Q) $),  \n1 = {veclen(\x1,\y1)},
  \p2 = ($ (N) - (Q) $),  \n2 = {veclen(\x2,\y2)} in
  (Q) circle ({sqrt(\n1)*sqrt(\n2)});
\coordinate[name intersections={of=circle_Q and PQ, by={B}}];
\coordinate[name intersections={of=circle_Q and circle_O, by={H,T}}];

\path[name path=circle_K] (K) circle ({sqrt(0.3*2.3)*\x});
\coordinate[name intersections={of=circle_P and circle_K, by={X,C}}];
\coordinate[name intersections={of=circle_Q and circle_K, by={X,D}}];
\coordinate[name intersections={of=circle_O and circle_K, by={E,F}}];

\end{pgfinterruptboundingbox}

\draw[color=cyan, thin] (180:\x) to (0:\x);

\pgfmathanglebetweenpoints{\pgfpointanchor{P}{center}}{\pgfpointanchor{G}{center}}
\pgfmathsetmacro{\anglePG}{\pgfmathresult}
\draw[name path=arc_GS, color=cyan, thin]
  let \p1 = ($ (G) - (P) $),  \n1 = {veclen(\x1,\y1)} in
  (G) arc (\anglePG:0:\n1);
\draw[name path=HT, color=cyan, thin]
  let \p1 = ($ (G) - (P) $),  \n1 = {veclen(\x1,\y1)} in
  (H) arc ({180-\anglePG}:{180}:\n1);
 
\pgfmathanglebetweenpoints{\pgfpointanchor{K}{center}}{\pgfpointanchor{E}{center}}
\pgfmathsetmacro{\angleKE}{\pgfmathresult}
\draw[name path=arc_EF, color=cyan, thin]
  (E) arc (\angleKE:540-\angleKE:{sqrt(0.3*2.3)*\x});

\pgfmathanglebetweenpoints{\pgfpointanchor{K}{center}}{\pgfpointanchor{C}{center}}
\pgfmathsetmacro{\angleKC}{\pgfmathresult}
\draw[name path=arc_CD, color=red, thick]
  (C) arc (\angleKC:540-\angleKC:{sqrt(0.3*2.3)*\x});

\pgfmathanglebetweenpoints{\pgfpointanchor{P}{center}}{\pgfpointanchor{C}{center}}
\pgfmathsetmacro{\anglePC}{\pgfmathresult}
\draw[name path=arc_AC, color=red, thick]
  let \p1 = ($ (C) - (P) $),  \n1 = {veclen(\x1,\y1)} in
  (C) arc (\anglePC:0:\n1);
\draw[name path=arc_BD, color=red, thick]
  let \p1 = ($ (G) - (P) $),  \n1 = {veclen(\x1,\y1)} in
  (1.67,2.31) arc ({180-\anglePC}:180:\n1);

\draw[color=red, thick] (A) to (B);
\draw[color=cyan] (O) to (0,{(1.3-sqrt(0.3*2.3))*\x});

\node[above] at (0.5,0) {$\widetilde{c}$};
\node[above=0.1, anchor=center] at ($(0,1.25)!0.3!(D)$) {$\widetilde\delta$};
\node[left=0.2, anchor=center] at ($(A)!0.5!(C)$) {$\widetilde{\eta}_1$};
\node[right=0.2, anchor=center] at ($(B)!0.5!(1.67,2.31)$) {$\widetilde{\eta}_2$};


\path (A) node[circle, fill, inner sep=1pt, label={[shift={(235:0.45)}, anchor=center]{$\widetilde{y}_1$}}]{};
\path (B) node[circle, fill, inner sep=1pt, label={[shift={(305:0.45)}, anchor=center]{$\widetilde{y}_2$}}]{};
\path (C) node[circle, fill, inner sep=1pt, label={[shift={(230:0.35)}, anchor=center]{$\widetilde{x}_1$}}]{};
\path (1.67,2.31) node[circle, fill, inner sep=1pt, label={[shift={(310:0.35)}, anchor=center]{$\widetilde{x}_2$}}]{};

\path (O) node[circle, fill, inner sep=1pt, label={[shift={(270:0.3)}, anchor=center]{$O$}}]{};

\end{tikzpicture}
\caption{\label{fig:boundary_length_annulus}
A covering of the annulus.}
\end{figure}

\begin{lemma}\label{winding002}
 If $c$ is a cusp with boundary length 4(it is embedded by Lemma \ref{lemma:cusp_maximal_area}), then 
 $$\ell(\delta)=2\log\bracket{2W(\delta)+\sqrt{4W^2(\delta)+1}}$$
\end{lemma}

\begin{proof}
When $c$ is a cusp, consider the projection map $p$ from the upper half plane model of $\bbh^2$ to $\Sigma$, where $p^{-1}(N(c))=\{(x,y)\in\bbh^2:y\geqslant 1\}$, and $A(-2,0)$ and $B(2,0)$ be two adjacent points of a same point $x\in\partial N(c)$, $\ell$ be the line $\{(x,y)\in\bbh^2:y=1\}$. Without loss of generality assume $\widetilde{\delta}$ is a lift of the arc $\delta$ and is an arc of a circle centered at the origin $O$ with endpoints $P_1,P_2\in\ell$. The hyperbolic length of the arc $P_1P_2$ is $d(P_1,P_2)=\ell(\delta)$, and $d_E(P_1,P_2)=4W(\delta)$ where $d_E$ means the Euclidean distance in $\bbh^2$, as in Figure \ref{fig:nbhd_cusps222} .Hence 
$$\ell(\delta)=d(P_1,P_2)=2\log\bracket{2W(\delta)+\sqrt{4W^2(\delta)+1}}$$
\end{proof}

\begin{figure}[htbp]
\begin{center}
\tikzexternaldisable
\begin{tikzpicture}
\def\x{1.5}

\draw[thick] (-1,0) to (-1,5);
\draw[thick] (1,0) to (1,5);

\draw[dashed] (4,0) arc (0:180:4);
\draw[dashed] (-4.5,1) to (4.5,1);

\draw[dashed] (-3*\x,0) to (3*\x,0);

\draw[thick, color=red] (3.873,1) arc (14.4775:165.5225:4);

\path (0,0) node[circle, fill, inner sep=1pt, label={[shift={(270:0.45)}, anchor=center]{$O$}}]{};
\path (-1,1) node[circle, fill, inner sep=1pt, label={[shift={(135:0.45)}, anchor=center]{$A$}}]{};
 \path (1,1) node[circle, fill, inner sep=1pt, label={[shift={(45:0.45)}, anchor=center]{$B$}}]{};
\path (-3.873,1) node[circle, fill, inner sep=1pt, label={[shift={(135:0.45)}, anchor=center]{$P_1$}}]{};
 \path (3.873,1) node[circle, fill, inner sep=1pt, label={[shift={(45:0.45)}, anchor=center]{$P_2$}}]{};
\path (3,3) node[circle, inner sep=1pt, label={[shift={(135:0.45)}, anchor=center]{$\widetilde\delta$}}]{};

\draw (0,5) node[above] {$\infty$};

	  			
\end{tikzpicture}
\end{center}
\tikzexternalenable
\caption{\label{fig:nbhd_cusps222}
A covering of $N_0(c_i)$}
\end{figure}

\section{Proof of the main theorem}\label{proof_main}

The notations are same as above, $\Gamma$ is a closed geodesic in $\Sigma$ with at least 2 self intersections, with length $\ell(\Gamma)$. Let $|\Gamma\cap\Gamma|$ be the self intersection number of $\Gamma$. Let $\gamma$ be the shortest arc contained in $\Gamma$ that is a closed curve in $\Sigma$, and $P\in\Sigma$ be its endpoint. Then $\gamma$ is simple since the length of $\delta_0$ is minimal. We finish the proof based on whether the length $\ell(\gamma)$ of $\gamma$ satisfies
$$\ell(\gamma)\leqslant2\log(5+2\sqrt{6})-4\log(1+\sqrt2)<1.06.$$
Briefly, if $\ell(\gamma) \geqslant 1.06$, then the remaining part $\Gamma \setminus \gamma$ is freely homotopic to a simple closed geodesic, and thus $\Gamma$ has at most one self-intersection. Otherwise, we can prove that $\Gamma$ lies in a pair of pants, and the conclusion can be derived using the results from Section \ref{geodesic_pants}. The detailed proof is presented below.

\subsection{Case 1: $\ell(\gamma)\geqslant1.06$}

\begin{theorem}\label{LG106}
 If $\ell(\gamma)\geqslant1.06$ and $\ell(\Gamma)<2\log(5+2\sqrt{6})$, then $|\Gamma\cap\Gamma|\leqslant1$.
\end{theorem}

\begin{proof}
 Let $\gamma'=\Gamma\setminus\gamma$, then $\gamma'$ is a geodesic arc which closed up and $\ell(\gamma')<4\log(1+\sqrt2)$. 
 Using Yamada's result(\cite{Y2799}), that is the infimum length of nonsimple closed geodesic on hyperbolic surfaces is $4\log(1+\sqrt2)$, we have $\gamma$ and $\gamma'$ are freely homotopic to a multiple of simple closed geodesics(or cusp) $\beta$ and $\beta'$, and
$$\ell(\beta)\leqslant\ell(\gamma)\qquad\ell(\beta')\leqslant\ell(\gamma')<4\log(1+\sqrt2)<3.5255$$
Using the collar lemma(Lemma \ref{collar}) we have $N(\beta)$ is an embedded annulus.
\begin{enumerate}
\item 

If $\beta\cap\beta'=\emptyset$, then since $\gamma'$ freely homotopic to a multiple of a simple closed geodesic $\beta'$ with multiplicity $k\in\mathbb{Z}_+$, denoted as $(\beta')^k$. We can conclude that $\gamma'$ is homotopic to a curve with basepoint $P$ with self-intersection number $k-1$, assume $\gamma''$ is the curve of minimal length satisfying this property. Then $\gamma''$ is a geodesic arc with the same endpoint as $P$, and it is freely homotopic to a multiple of $\beta'$ and the union of $\gamma''$ and $\beta'$ could not create bigons. We can also conclude that $\gamma''\cap\beta'=\emptyset$. Hence $\Gamma$ is freely homotopic to a curve starting at $P$, winding around $\beta$ once, then moving to $\beta'$ and winding around it $k$ times, before finally returning to $P$. This curve has a self-intersection number of $k$.

We can define a pair of pants $P$ with $\beta$ and $\beta'$ as two of three boundary components in the following manner. Connect $P$ and $\beta$ using the shortest geodesic $\delta_1$ in the annulus bounded by $\beta$ and $\gamma$, and connect $P$ and $\beta'$ using the shortest geodesic $\delta_2$ in the annulus bounded by $\beta'$ and a portion of $\gamma'$. Consider a closed curve $\beta_1:=\delta_1\beta\delta_1^{-1}\delta_2\beta'\delta_2^{-1}$ (by connecting the endpoints) we can establish an orientation of $\beta$ and $\beta'$ such that $\beta_1$ is freely homotopic to a simple closed geodesic(or cusp) $\beta''$, which serves as the third boundary component of $P$. As the homotopy type of $\Gamma$ can be represented by a closed geodesic $\Gamma'$ contained in $P$, the uniqueness of closed geodesics with a given homotopy type implies that $\Gamma=\Gamma'\subseteq P$. Using Corollary \ref{pants2} we have $\ell(\Gamma)\geqslant2\log(5+2\sqrt6)$.

\item If $\beta\cap\beta'\neq\emptyset$, then $\beta$ is a simple closed geodesic, not a cusp. Assume $t=\frac12\ell(\beta)>0$. Since $\gamma$ and $\beta$ freely homotopic, they bound an annulus $A$, and $\gamma$ lies on one side of $\beta$, $N'(\beta)$ is the half-collar on that side. First note that if $\beta\cap\beta'\neq\emptyset$, since every arc $\beta'''$ of $\beta'\cap N_1(\beta)$ satisfying $\beta'''\cap\beta=\emptyset$ is a geodesic segment passing through $N_1(\beta)$, we have
$$\ell(\gamma')\geqslant\ell(\beta')\geqslant\ell(\beta''')\geqslant2w(\ell(\beta))=2\log\frac{e^t+1}{e^t-1}$$
Let $\widetilde\gamma\subseteq\Gamma$ be the unique arc in $\Gamma\cap N'(\beta)$ containing or contained in $\gamma'$. The winding number $W(\widetilde\gamma)$ can be defined in the beginning of Section \ref{winding_number}.

If $W(\widetilde\gamma)\geqslant1$, then using Theorem \ref{winding001} we have
\begin{align*}
\ell(\Gamma\setminus{\beta'''})&\geqslant\ell(\widetilde\gamma)=2\sinh^{-1}\bracket{\sinh (tW(\widetilde\gamma))\cosh w_1(\ell(\beta))}\\
 &\geqslant2\sinh^{-1}\bracket{\sinh t\cosh\log\frac{e^{t/2}+1}{e^{t/2}-1}}
\end{align*}
If $W(\widetilde\gamma)<1$, then in annulus $A$, $d(P,\beta)>w_1(\ell(\beta))$, hence
\begin{align*}
\ell(\Gamma\setminus{\beta'''})&\geqslant\ell(\gamma)=2\sinh^{-1}\bracket{\sinh t\cosh d(P,\beta)}\\
 &\geqslant2\sinh^{-1}\bracket{\sinh t\cosh\log\frac{e^{t/2}+1}{e^{t/2}-1}}\\
 &=2\log\bracket{\frac{(e^t+1)^2}{2e^t}+\sqrt{\bracket{\frac{(e^t+1)^2}{2e^t}}^2+1}}
\end{align*}
In both two cases, since $\beta\cap\beta'\neq\emptyset$, we have $\ell(\beta)>0$, then $T=\frac{(e^t+1)^2}{2e^t}>2$, hence we have
$$\ell(\Gamma)\geqslant\ell(\beta''')+\ell(\Gamma\setminus\beta''')\geqslant2W(\ell(\beta))+\ell(\Gamma\setminus\beta''')\geqslant H$$
where 
\begin{align*}
H=&2\log\frac{e^t+1}{e^t-1}+2\log\bracket{\frac{(e^t+1)^2}{2e^t}+\sqrt{\bracket{\frac{(e^t+1)^2}{2e^t}}^2+1}}\\
=&\log\frac{T}{T-2}+2\log\bracket{T+\sqrt{T^2+1}}
\end{align*}
\begin{align*}
\frac{dH}{dT}=\frac2{\sqrt{T^2+1}}-\frac2{T(T-2)}
\end{align*}
There exists $T_0>2$ such that when $2\leqslant T\leqslant T_0$, $\frac{dH}{dT}\leqslant0$, and when $T\geqslant T_0$, $\frac{dH}{dT}\geqslant0$. When $T=3$ $\frac{dH}{dT}<0$, and when $T=\frac{25}8$ $\frac{dH}{dT}>0$, hence $3<T_0<\frac{25}8$. Hence for $T>2$ we have a contradiction by
$$H(T)\geqslant H(T_0)>\log\frac{25}{9}+2\log\bracket{3+\sqrt{10}}>2\log(5+2\sqrt6)$$

\end{enumerate}
\end{proof}

\subsection{Case 2: $\ell(\gamma)<1.06$}

The idea of the proof is to find the homotopy type of $\Gamma$ and prove that $\Gamma$ is in a pair of pants, then we use the conclusions in Section 3 to finish the proof. Notation as before, let $\beta$ be the simple closed geodesic or cusp free homotopic to $\gamma$.

\begin{theorem}\label{LL106}
 If $\ell(\gamma)<1.06$ and $\ell(\Gamma)<2\log(5+2\sqrt6)$, then we can find a pair of pants $P\subseteq\Sigma$, such that the interior of $P$ is embedded in $\Sigma$, each boundary component is a simple closed geodesic or a cusp, and $\beta$ is one of the boundary components.
\end{theorem}

\begin{proof}
 Choose the generalized collar $N_1(\beta)$ such that $\gamma$ and the bigger halfcollar $N'(\beta)$ is on the same side of $\beta$. Since $\ell(\gamma)<1.06<2\log{2+\sqrt5}$, and when $\beta$ is a closed geodesic, we have $1.06<2\sinh^{-1}\bracket{\sinh\frac{\ell(\beta)}{2}\cdot\cosh w_1(\ell(\beta))}$, by Lemma \ref{winding001} and Lemma \ref{winding002}, we have $\gamma\subseteq N'(\beta)$. Let $\widetilde{\gamma}$ be the arc in $\Gamma\cap N_1(\beta)$ containing $\gamma$. Assume $|\widetilde{\gamma}\cap\widetilde{\gamma}|=k$, and the endpoints of $\widetilde{\gamma}$ are $y_1,y_2\in\partial N_1(\beta)$. Let $\alpha\in(k,k+1]$ be the winding number of $\widetilde{\gamma}$, and assume $\widetilde{\gamma}\cap\widetilde{\gamma}=\{P=P_1,P_2,...,P_k\}$, and in $N_1(\beta)$, $d(P_i,\beta)$(or $d(P_i,h(\beta))$ when $\beta$ is a cusp and $h(\beta)$ is a sufficiently small horocycle) is increasing on $i\in\{1,...,k\}$. Let $\zeta\subseteq N_1(\beta)$ be the unique simple geodesic with endpoints $y_1,y_2$ and homotopic to the curve $\widetilde\zeta$ starting from $y_1$ and goes along $\widetilde{\gamma}$ to $P_k$, then goes along $\widetilde{\gamma}$ from $P_k$ to $y_2$. By definition $\zeta$ has winding number $\alpha-k\in(0,1]$.
\begin{enumerate}
\item If $\beta$ is a closed geodesic, assume $\ell(\beta)=2t>0$, let $$H_1(s,t)=\sinh^{-1}\bracket{\sinh{(st)}\cosh(w_1(2t))}$$ 
Let $u=2(\cosh \frac{t}2)^2$ we have
\begin{align*}
 \ell(\widetilde{\gamma})-\ell(\zeta)=&2H_1(k+\alpha,t)-2H_1(\alpha,t)\\
 \geqslant &2H_1(1+\alpha,t)-2H_1(\alpha,t)
 \geqslant 2H_1(2,t)-2H_1(1,t)\\
=&2\sinh^{-1}\bracket{\sinh{2t}\cosh(w_1(2t))}\\
&-2\sinh^{-1}\bracket{\sinh{t}\cosh(w_1(2t))}\\
=&2\sinh^{-1}(u(2u-2))-2\sinh^{-1}(u)\\
\geqslant&2\sinh^{-1}(2u)-2\sinh^{-1}u\\
\geqslant&2\sinh^{-1}4-2\sinh^{-1}2>1.06\\
>&2\log(5+2\sqrt6)-4\log(1+\sqrt2)
\end{align*}
The second inequality using the fact that $H_1(\alpha,t)$ is a concave function on $\alpha$ by taking second derivative of $\alpha$. The third inequality holds since $u>2$. The fourth inequality using $2\sinh^{-1}(2u)-2\sinh^{-1}u$ is an increasing function on $u>2$. Hence we have $$\ell(\Gamma\setminus\widetilde\gamma)+\ell(\zeta)<4\log(1+\sqrt2)$$ 
By Yamada's result, the closed curve $(\Gamma\setminus\widetilde\gamma)\cup\zeta$ is freely homotopic to a multiple of a simple closed geodesic, say $\beta'$, and so is $(\Gamma\setminus\widetilde\gamma)\cup\widetilde\zeta$. Hence if $\beta\cap\beta'=\emptyset$, $\Gamma$ is freely homotopic to a curve starting at $P_k$, winding around $\beta$ finitely many times, then goes to $\beta'$ then winding around finitely many times, finally goes back to $P_k$. Similar as case (1) in the proof of Theorem \ref{LG106}, the theorem holds. If $\beta\cap\beta'\neq\emptyset$ then same as case (2) in Theorem \ref{LG106} to finish the proof.
\item If $\beta$ is a cusp, similar as the previous case, let $H_1(\alpha,t)=\log(2\alpha+\sqrt{4\alpha^2+1})$ and $\ell(\widetilde{\gamma})-\ell(\zeta)=2H_1(k+\alpha,t)-2H_1(\alpha,t)$ also holds, similarly we can prove the case.
\end{enumerate}
\end{proof}

\begin{proof}[Proof of Theorem \ref{main}]
 If $\ell(\gamma)\geqslant1.06$, then Theorem \ref{LG106} implies that $\ell(\Gamma)\geqslant2\log(5+2\sqrt6)$. If $\ell(\gamma)<1.06$, then Theorem \ref{LL106} implies that $\Gamma$ lies in a pair of pants and using Corollary \ref{length_pants} we get the conclusion.
\end{proof}

\begin{remark}
    For \(k > 2\), the proof strategy remains fundamentally the same, although the details diverge as the geometry of short geodesics becomes more complex. For small \(k\), the types of closed geodesics with length less than \(2\cosh^{-1}(2k+1)\) admit a complete classification into a manageable number of types. This classification strongly suggests that \(M_k\) is computable exactly in these cases. 
    
    For larger values of \(k\), we have a roadmap for obtaining sharp bounds: following the approach in \cite{Y2806, Y2800}, a thick-thin decomposition can be employed. The behavior of geodesics in the thin parts is explicitly characterizable, while their intersections in the thick parts can be estimated from below using the injectivity radius. We are confident that a refined optimization of this framework will yield improved bounds on \(M_k\) for a substantial range, such as \(3 \leq k < 1750\). Furthermore, the same methodology also opens up the possibility of analyzing other characteristics of geodesics on surfaces.
\end{remark}

\end{document}